\documentclass[reqno,final]{amsart}
\usepackage{natbib}  
\usepackage{fancyhdr} 
\usepackage{color} 
\usepackage{hyperref} 
\usepackage{graphicx} 

\usepackage{bbm}
\usepackage[font=footnotesize,skip=0pt]{caption}
\usepackage{tikz}
\usetikzlibrary{decorations.pathreplacing}
\usetikzlibrary{backgrounds}

\definecolor{aleacolor}{rgb}{0.16,0.59,0.78}

\hypersetup{
breaklinks,
colorlinks=true,
linkcolor=aleacolor,
urlcolor=aleacolor,
citecolor=aleacolor}

\renewcommand{\headheight}{11pt}
\renewcommand{\cite}{\citet}

\theoremstyle{plain}
\newtheorem{theorem}{Theorem}[section]                                          
                          
\newtheorem{lemma}[theorem]{Lemma}
\newtheorem{corollary}[theorem]{Corollary}

\theoremstyle{definition}

\theoremstyle{remark}
\newtheorem{remark}[theorem]{Remark}

\makeatletter \@addtoreset{equation}{section} \makeatother
\renewcommand\theequation{\thesection.\arabic{equation}}
\renewcommand\thefigure{\thesection.\arabic{figure}}
\renewcommand\thetable{\thesection.\arabic{table}}

\newcommand{\R}{\mathbb R}

\newcommand{\Z}{\mathbb Z}
\newcommand{\N}{\mathbb N}

\begin{document}

\title{On the survival probability of a random walk in random environment with killing}

\author{Stefan Junk}

\address{Technische Universit\"at M\"unchen\newline
Department of Mathematics\newline
M14 - Chair for Probability Theory\newline
Boltzmannstr. 3,\newline
85747 Garching bei M\"unchen, Germany}

\email{junk@tum.de}
\urladdr{\url{https://www-m14.ma.tum.de/personen/junk/}}


\begin{abstract}We consider one dimensional random walks in random environment where every time the process stays at a location, it dies with a fixed probability. Under some mild assumptions it is easy to show that the survival probability goes to zero as time tends to infinity. In  this paper we derive formulas for the rate with which this probability decays. It turns out that there are three distinct regimes, depending on the law of the environment.
\end{abstract}

\maketitle

\section{Introduction}
We use the following notations to describe random walks in random environments (RWREs) on $\mathbb{Z}$ in i.i.d. environments: Let $P$ be a probability measure on
\[
\Omega:=\{(\omega_x)_{x\in\Z}=(\omega^+_x,\omega^0_x,\omega^-_x)_{x\in\Z} \ | \forall x:\ \omega_x^+,\omega_x^0,\omega_x^-\geq 0,\omega_x^++\omega_x^0+\omega_x^-=1\}
\]
We interpret an element $\omega\in\Omega$ as the transition probabilities for a random walk in $\Z$: Let $(X_n)_{n\in\mathbb{N}}$ be Markov chain on $\mathbb{Z}$ with transition kernel 
\[
\begin{array}{cc}
P_\omega(X_{n+1}=z+1|X_n=z)& :=\omega_z^+ \\ P_\omega(X_{n+1}=z-1|X_n=z)& :=\omega_z^- \\ P_\omega(X_{n+1}=z|X_n=z)& :=\omega_z^0
\end{array}
\]
We write $P_\omega^z:=P_\omega(.|X_0=z)$, and will omit the $z$ if the random walk starts in zero. We are interested in the joint probability measure 
\[\mathbb{P}^z:=P\times P^z_\omega\] 
of environment and random walk. 
An important assumption will be that $P$ is a product measure, meaning that 
\[\{\omega_x:x\in\mathbb{Z}\}\text{ is i.i.d. under }P\]
It is reasonable to demand that, $P$-almost surely, $\omega_0^+>0$ and $\omega_0^->0$, so that (with probability one) the process can move infinitely far in both directions. We make the even stronger assumption, that ($P$-almost surely) there is a uniform bound for the minimal probability of going to the left or right. This condition is usually called "uniform ellipticity" (\ref{eq:defue}):
\begin{equation}\label{eq:defue}\tag{UE}
\exists\varepsilon_0>0\quad\text{ s.t. }P(\omega_0^+\geq\varepsilon_0)=P(\omega_0^-\geq\varepsilon_0)=1
\end{equation}
An important quantity is denoted
\begin{equation}\label{eq:rho}
\rho_i:=\frac{\omega_i^-}{\omega_i^+} \text{ for } i \in\mathbb{Z}
\end{equation}
An overview over results for this kind of random walk in random environment can be found for example in \cite{rwre}.\par We now introduce the model we are interested in. It is originally motivated by two papers on statistical mechanics\footnote{\cite{phy1} and \cite{phy2}} where a higher-dimensional version was used to describe polymers folding in a solution having random variations in the local density. From this physical motivation it is desirable that the RWRE stays in parts of the environment where the probability on staying at the same location (which can be interpreted as a local density of a solution) is small. We model this by choosing some $r\in(0,1)$, and whenever the process stays in place, it dies with probability $r$. 
\begin{remark}
The results in this paper are for measures $P$ on the environment where the survival probability is dominated by events depending only on $P$, such that the parameter $r$ does not appear in the result. See also the remark in section 3.5.
\end{remark}
Formally, consider a probability space as above where we additionally have a sequence $(\xi_n)_n$ of i.i.d. Bernoulli random variables, independent of environment and random walk, with success probability $r$. Then we can define the extinction time $\tau$ by
\begin{equation}
\tau:=\inf\{n\geq 1:X_n=X_{n-1},\xi_n=1\}
\end{equation}
If we assume that there is a positive probability for extinction, that is
\[P(\omega_0^0>0)>0\]
then it is easy to show $\lim_{n\to\infty}\mathbb{P}(\tau>n)= 0$. We are interested in the asymptotic behavior of $\mathbb{P}(\tau>n)$. 
\section{Preliminaries}
One can easily give a lower bound for the survival probability by considering a set of environments in which the survival probability is large, and such that the probability for such a favorable environment is not too small. In this section we introduce the notion of a valley, which will play the role of such an environment. \par For some fixed $\omega\in\Omega$ we define the potential function as follows:
\begin{equation}\label{eq:defV}
V:\R\to\R,\quad x\mapsto\left\{\begin{matrix}
\sum_{i=0}^{\lfloor x\rfloor}\ln\rho_i & \text{ if }{\lfloor x\rfloor}>0\\
0&\text{ if }{\lfloor x\rfloor}=0\\
-\sum_{i={\lfloor x\rfloor}}^{-1}\ln\rho_i&\text{ if }{\lfloor x\rfloor}<0
\end{matrix}\right.
\end{equation}
\\For some interval $[a,c]$ and an environment $\omega$ we define the following quantities:
\[H_+(a,c):=\max_{b\in[a,c]}\left(\max_{x\in[b,c]} V(x)-\min_{x\in[a,b]}V(x)\right)\]
\[H_-(a,c):=\max_{b\in[a,c]}\left(\max_{x\in[a,b]} V(x)-\min_{x\in[b,c]}V(x)\right) \]
\[H(a,c):=\min\left\{H_+(a,c),H_-(a,c)\right\}\]
See also figure~\ref{figure} for an illustration. When no confusion occurs we simply write $H_+,H_-$ and $H$. 
\begin{figure}[h]
\centering
\pgfmathsetseed{3}
\newcommand{\Emmett}[4]{
\draw[#4] (0,0)
\foreach \x in {1,...,#1}
{   -- ++(#2*2,rand*#3*2)
};
}
\begin{tikzpicture}[scale=1.5]
\Emmett{80}{0.02}{0.2}{black}

\node[inner sep=0.7pt, outer xsep=0pt] (n1) at (0.1,0.44){};
\node[inner sep=0.7pt, outer xsep=0pt] (n2) at (-0.5,0.44){};
\node[inner sep=0.7pt, outer xsep=0pt] (n3) at (-0.5,-2.22){};
\node[inner sep=0.7pt, outer xsep=0pt] (n4) at (3,-2.22){};
\node[inner sep=0.7pt, outer xsep=0pt] (n5) at (2.92,0.86){};
\node[inner sep=0.7pt, outer xsep=0pt] (n6) at (3.5,0.86){};
\node[inner sep=0.7pt, outer xsep=0pt] (n7) at (3.5,-2.22){};
\node[inner sep=0.7pt, outer xsep=0pt] (n8) at (0,-2.22){};
\path[draw] (n1) -- (n2);
\path[draw] (n3)-- (n7);
\path[draw] (n5)-- (n6);
\path[draw] [thick,decorate,decoration={brace,amplitude=7pt,mirror},xshift=0pt,yshift=-40pt]
(n2) -- (n3)node[midway, color=black,xshift=-17pt] {$H_-$};
\path[draw] [thick,decorate,decoration={brace,amplitude=7pt},xshift=0pt,yshift=-40pt]
(n6) -- (n7)node[midway, color=black,xshift=17pt] {$H_+$};

\draw[->,thick] (-1,-2.4) -- (4.5,-2.4) node[above] {$x$};
\draw[->,thick] (0.4,-2.6) -- (0.4,1) node[right] {$V(x)$};
\draw (1.56,-2.3)--(1.56,-2.5) node[below] {$b$};
\draw (3.2,-2.3)--(3.2,-2.5) node[below] {$c$};
\draw (0,-2.3)--(0,-2.5) node[below] {$a$};
\end{tikzpicture}
\caption{We denote by $H_-$ the maximal difference $V(x)-V(y)$ in the potential between any two points $x<y$ in $[a,b]$. The same holds for $H_+$ with $x<y$ replaced by $x>y$. Starting from the point of minimal potential, the random walk has to overcome a potential difference of at least $H_-\land H_+=H$ to leave the valley.}
\label{figure}
\end{figure}
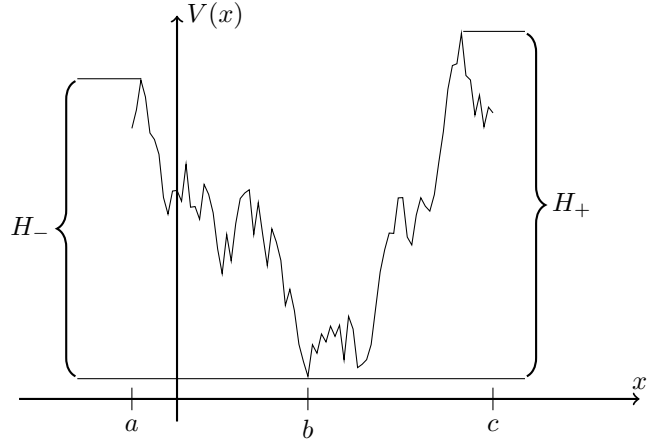
\par We denote the first hitting time of the boundary by \[U=U_{a,c}:=\inf\{n\geq 0:X_n=a\vee X_n=c\}\]
Then we have the following two lemmas, which give upper and lower bounds on the probability of leaving a valley.
\begin{lemma}\label{lemmapot1}
Let $\omega\in\Omega$ be an environment such that
\[\forall x\in\Z:\quad \omega_x^0=0, \omega_x^+\geq \varepsilon_0,\omega_x^-\geq \varepsilon_0\]
where $\varepsilon_0$ is the constant from (\ref{eq:defue}). Then there exist constants $\gamma_1>0$ and $\gamma_2=\gamma_2(\varepsilon_0)>0$ such that for $c-a\geq \gamma_2(\varepsilon_0)$ and for all $n\geq 1$ we have:
\begin{equation}\label{eq:conf1}
\max_{x\in(a,c)}P_\omega^x\left(\frac{U_{a,c}}{\gamma_1(c-a)^4e^{H(a,c)}}>n\right)\leq e^{-n}
\end{equation}
\end{lemma}
\begin{lemma}\label{lemmapot2}Let $\omega$ be as in lemma \ref{lemmapot1}, and assume we find some $b\in(a,c)$ such that 
\begin{equation}\label{lemmapotvor1}V(c)=\max_{x\in[b,c]}V(x), \ \ V(a)=\max_{x\in[a,b]}V(x)\end{equation}
Then there are $\gamma_3,\gamma_4>0$ such that for $c-a\geq \gamma_4$ and for all $n\geq 1$ we have
\begin{equation}
\min_{x\in(a,c)}P_\omega^x\left(\gamma_3\ln(2(c-a))\frac{U_{a,c}}{e^{H(a,c)}}>n\right)\geq\frac{1}{2(c-a)}e^{-n}
\end{equation}
\end{lemma}
This version of lemmas \ref{lemmapot1} and \ref{lemmapot2} is taken from \cite{soft}\footnote{\;cf. \cite{soft}, p 23}, whereas the proofs can be found in \cite{soft2}\footnote{\;cf.  \cite{soft2}, pp 15}. Note that the point $b$ of minimal potential appears only in the assumption, but not in the result of lemma~\ref{lemmapot2}. But if we restrict ourselves to a random walk starting from $b$, we can get rid the assumption that the environment attains maximal potential at the edges: 
\begin{corollary}\label{korrpot}
Under the assumptions of lemma \ref{lemmapot1}, let $b$ be the point with minimal potential in $(a,c)$. Then for all $c-a\geq \gamma_4$ and all $n\geq 1$ we have
\begin{equation}\label{eq:conf2} P_\omega^b \left(\gamma_3\ln(2 (c-a))\frac{U_{a,c}}{e^{H(a,c)}}>n\right)\geq\frac{1}{2(c-a)}e^{-n}\end{equation}
Here we can use the same constants $\gamma_3,\gamma_4$ as in lemma \ref{lemmapot2}.
\end{corollary}
\begin{proof}[Proof of the corollary]
Choose $\bar{a}\in [a,b]$ and  $\bar{c}\in [b,c]$ such that $V(\bar{a})=\max_{x\in[a,b]} V(x)$ and $V(\bar{c})=\max_{x\in[b,c]} V(x)$. Then $[\bar{a},\bar{c}]$ satisfies (\ref{lemmapotvor1}), and we have $H(a,c)=H(\bar{a},\bar{c})$. Using lemma \ref{lemmapot2} we get
\[P_\omega^b \left(\gamma_3\ln(2 (c-a))\frac{U_{a,c}}{e^{H(a,c)}}>n\right)\geq P_\omega^b \left(\gamma_3\ln(2 (\bar{c}-\bar{a})) \frac{U_{\bar{a},\bar{c}}}{e^{H(\bar{a},\bar{c})}}>n\right)\geq\]
\[\geq\frac{1}{2(\bar{c}-\bar{a})}e^{-n}\geq\frac{1}{2(c-a)}e^{-n}\]
\end{proof}
Let $b_1,b_2,h>0$, $x\in\Z$ and $k\in\N$. We define the event that there is a valley of depth $h\ln n$ at the interval $I_n:= [x-b_1\ln n,x+b_2\ln n]$ around the location $x$ by
\begin{equation}\label{eq:defS_n}
S_n^k(x,b_1,b_2,h):=
\left\{\forall i\in I_n:\omega_i^0\leq\frac{1}{k} \right\}\cap
\left\{\begin{matrix}
V(x)=\min_{y\in I_n}V(y)\\[1mm]
V(x)-V(x-b_1\ln n)\geq h\ln n\\[1mm]
V(x-b_2\ln n)-V(x)\geq h\ln n\\
\end{matrix}\right\}
\end{equation}
Note that since we extended the definition of $V$ to $\R$, this depends only on locations in 
\[\big[x+\lfloor -b_1\ln n\rfloor,x+\lfloor b_2\ln n\rfloor\big]\cap \Z\]
In the definition of $S_n^k$, the left event ensures that while the random walk stays inside $I_n$, the probability of dying is not too large. The right event implies for $\omega\in S_n^k(x,b_1,b_2,h)$ that 
\[H(x-b_1\ln,x+b_2\ln)\geq h\ln x\]
so that we can use corollary \ref{korrpot} to bound the probability of leaving $I_n$. 
\par For $k=\infty$, we denote
\[S^\infty_n(x,b_1,b_2,h):=\bigcap_{k\in\N} S^k_n(x,b_1,b_2,h)\]
Equivalently, $S^\infty(x,b_1,b_2,h)$ can be defined as in (\ref{eq:defS_n}), with the left event replaced by $\left\{\forall i\in I_n:\omega_i^0=0 \right\}$. 
\section{Results}
\subsection{The polynomial case}
We first cover the case where we can create a valley such that $\omega_\cdot^0$ is zero on the inside of the valley. Consequently, the random walk survives as soon as we can ensure that it does not leave the valley.
\begin{lemma}\label{lemma1}Assume 
\begin{equation}\label{eq:lemma1vor1}
p:=\min\Big\{P(\ln\rho_0>0,\omega_0^0=0),P(\ln\rho_0<0,\omega_0^0=0)\Big\}>0
\end{equation}
Then for all $b_1,b_2,h>0$ and $k\in\N\cup\{\infty\}$, we have 
\[\lim_{n\to\infty}\frac{\ln P(S_n^k(x,b_1,b_2,h))}{-\ln n} = C_k(b_1,b_2,h)\]
where
\[C_k(b_1,b_2,h)=-(b_1+b_2)\ln P\left(\omega_0^0\leq\frac{1}{k}\right)+\sup_{t>0}\left\{ht-b_1\ln E\left(\rho_0^t\left|\omega_0^0\leq\frac{1}{k}\right.\right)\right\}+\]
\begin{equation}\label{eq:lemma1res1}
+\sup_{t>0}\left\{ht-b_2\ln E\left(\rho_0^{-t}\left|\omega_0^0\leq\frac{1}{k}\right.\right)\right\}
\end{equation}
Furthermore we set
\[D_k(h):=\inf\Big\{C_k(b_1,b_2,h):b_1,b_2>0\Big\}\]
Then
\[\lim_{k\to\infty} D_k(h)=D_\infty(h)\]
\end{lemma}
The main result of this paper is that in this case, $\mathbb P(\tau>n)$ decays at a polynomial rate.
\begin{theorem}\label{thm1}
Assumption (\ref{eq:lemma1vor1}) implies
\[\lim_{n\to\infty}\frac{\ln \mathbb{P}(\tau > n)}{-\ln n} = D_\infty(1)\in (0,\infty)\]
\end{theorem}
\subsection{Survival inside of a valley}
In this section we cover the case where (\ref{eq:lemma1vor1}) is violated, meaning there are no valleys with $\omega_\cdot^0=0$ on the inside. It may however happen that there are valleys consisting of locations that are not too dangerous in the sense that $\omega_\cdot^0$ decays with $n$. We denote such an event by
\begin{equation}\label{eq:defT_n}T_n(x,b_1,b_2,h):=S_n^n(x,b_1,b_2,h)\end{equation}
\[=\left\{\forall i\in I_n:\omega_i^0\leq\frac{1}{n} \right\}\cap
\left\{\begin{matrix}
V(x)=\min_{y\in I_n}V(y)\\[1mm]
V(x)-V(x-b_1\ln n)\geq h\ln n\\[1mm]
V(x-b_2\ln n)-V(x)\geq h\ln n\\
\end{matrix}\right\}\]
To decide when $T_n$ has positive probability we look at
\begin{align*}
p_n^+:=P\left(\ln\rho_0>0,\omega_0^0\leq\frac{1}{n}\right)\\
p_n^-:=P\left(\ln\rho_0<0,\omega_0^0\leq\frac{1}{n}\right)\\
p_n^0:=P\left(\ln\rho_0=0,\omega_0^0\leq\frac{1}{n}\right)
\end{align*}
By monotone convergence we see that (\ref{eq:lemma1vor1}) being violated implies 
\begin{equation}\label{eq:lemma2vor1}
\lim_{n\to\infty}\min\Big\{p_n^+,p_n^-\Big\}=0
\end{equation}
For this section assume 
\begin{equation}\label{eq:lemma2vor2}
\forall n: \ \min\Big\{p_{n}^+,p_{n}^-\Big\}>0
\end{equation}
Then $T_n$ has positive probability for all $n$. In order to compute $P(T_n)$ we need some regularity in the way $\ln\rho_\cdot$ behaves when conditioning on $\left\{\omega_\cdot^0\leq\frac{1}{n}\right\}$. We assume the following limits exist as weak limits of probability distributions:
\begin{subequations}\begin{align}P_n^+(\cdot):=P\left(\ln\rho_0\in\cdot\left|\omega_0^0\leq\frac{1}{n},\ln\rho_0> 0\right.\right)&\xrightarrow[n\to\infty]{w}&P^+(\cdot)\label{eq:convdistrp}\\
P_n^-(\cdot):=P\left(-\ln\rho_0\in\cdot\left|\omega_0^0\leq\frac{1}{n},\ln\rho_0< 0\right.\right)&\xrightarrow[n\to\infty]{w}&P^-(\cdot)&\label{eq:convdistrm}\\
Q_n^+(\cdot):=P\left(\ln\rho_0\in\cdot\left|\omega_0^0\leq\frac{1}{n},\ln\rho_0\geq 0\right.\right)&\xrightarrow[n\to\infty]{w}&Q^+(\cdot)\label{eq:convdistrpp}\\
Q_n^-(\cdot):=P\left(-\ln\rho_0\in\cdot\left|\omega_0^0\leq\frac{1}{n},\ln\rho_0\leq 0\right.\right)&\xrightarrow[n\to\infty]{w}&Q^-(\cdot)\label{eq:convdistrmm}\end{align}\end{subequations}
Here $\xrightarrow{w}$ denotes weak convergence, and $P^+, P^-, Q^+$ and $Q^-$ are the limiting measures having support in $[0,\infty)$. We make the following assumption to ensure that the first two limits are non-degenerate, meaning $(0,\infty)$ has positive probability. 
\begin{equation}\label{eq:nontriv}
\exists t>0\text{ such that }\quad P^+([t,\infty))>0 \text{ and }P^-([t,\infty))>0
\end{equation}
Note that $Q^+$ or $Q^-$ are allowed to be degenerate, by which we mean equal to the Dirac measure in zero. Also note that in the previous case we had
\[P^+(\cdot)=P(\ln\rho_0\in \cdot|\omega_0=0,\ln\rho_0> 0)\]
and (\ref{eq:lemma1vor1}) implied (\ref{eq:nontriv}). We need to look at the limiting distribution of $\ln\rho_\cdot$ still a little closer: Define 
\[\varepsilon_n^+:=\sup\left\{\varepsilon:P\left(\omega_0^0\leq\frac{1}{n},\ln\rho_0>\varepsilon\right)>0\right\}\]
\[=\sup\left\{\varepsilon:P\left(\ln\rho_0>\varepsilon\left|\omega_0^0\leq\frac{1}{n},\ln\rho_0>0\right.\right)>0\right\}\leq\ln \frac{1-\varepsilon_0}{\varepsilon_0}\]
Here the last inequality is due to (\ref{eq:defue}). The sequence $(\varepsilon_n^+)_n$ is decreasing and bounded by zero, therefore some limit exists:
\[\varepsilon^+:=\lim_{n\to\infty}\varepsilon_n^+\]
We have $P_n^+([\varepsilon_{m}^+,\infty))=0$ all $n\geq m$. Therefore $P^+([\varepsilon_{m}^+,\infty))=0$ for any $m$, and (\ref{eq:nontriv}) implies $\varepsilon^+>0$. Alternatively, $\varepsilon^+$ is the essential supremum of a random variable having distribution $P^+$. We define $\varepsilon^-\in(0,\infty)$ similarly as the essential supremum of a random variable with distribution $P^-$. In the same way the essential infima can be controlled:
\begin{equation}\label{eq:delta}
\delta_n^+:=\sup\Big\{t\geq 0:Q_n^+([0,t])=0\Big\}\geq 0
\end{equation}
This is an increasing sequence bounded by $\varepsilon^+$, and we set
\[\delta^+:=\lim_{n\to\infty}\delta_n^+\in[0,\epsilon^+]\]
Note that $\delta^+$ and $\delta^-$ correspond to the essential infima of random variables having distribution $Q^+$ and $Q^-$. Both $\delta^+$ or $\delta^-$ may be zero, for example if (\ref{eq:lemma2vor1}) holds together with $P(\omega_0^0=0,\omega_0^+=\omega_0^-=\frac{1}{2})>0$. Those quantities will now play a role in the value of the exponent.
\subsection{The intermediate case}
\begin{lemma}\label{lemma2}In addition to assumptions (\ref{eq:lemma2vor1}) - (\ref{eq:nontriv}), assume the following limits exist in $[0,\infty]$:
\[a^+:=\lim_{n\to\infty}\frac{\ln p_n^-}{\ln p_n^+} \quad\text{ and }\quad
a^+_0:=\lim_{n\to\infty}\frac{\ln p_n^0}{\ln p_n^+}\]
and set
\[a^-:=\lim_{n\to\infty}\frac{\ln p_n^+}{\ln p_n^-}= \left(a^+\right)^{-1}\quad\text{ as well as }\quad a^-_0:=\lim_{n\to\infty}\frac{\ln p_n^0}{\ln p_n^-}=a^- a_0^+\]
We then get
\begin{equation}\label{eq:lemma2res1}
\lim_{n\to\infty} \frac{\ln P(T_n(x,b_1,b_2,h))}{\ln n \ln(\min\{p_n^+,p_n^-\})} = C(b_1,b_2,h)
\end{equation}
where
\begin{equation}\label{eq:lemma2res2}C(b_1,b_2,h)=\min\{1,a^-\}C^+(b_2,h)+\min\{1,a^+\}C^-(b_1,h)\end{equation}
and
\begin{align*}
C^-(b_1,h)=\frac{h+\delta^+ b_1+ \min\{1,a^-,a^-_0\} (-h+b_1\varepsilon^-)}{\varepsilon^-+\delta^+}\\
C^+(b_2,h)=\frac{h+\delta^- b_2+ \min\{1,a^+,a^+_0\} (-h+b_2\varepsilon^+)}{\varepsilon^++\delta^-} 
\end{align*}
This result applies uniformly for all $b_1\in [\frac{h}{\varepsilon^-},k],b_2\in [\frac{h}{\varepsilon^+},k]$ and furthermore we have
\[\inf\left\{C^-(b_1,h):b_1\in \left[\frac{h}{\varepsilon^-},k\right]\right\}=\frac{h}{\varepsilon^-}\]
\[\inf\left\{C^+(b_2,h):b_2\in \left[\frac{h}{\varepsilon^+},k\right]\right\}=\frac{h}{\varepsilon^+}\]
\end{lemma}
In the lemma, we set $\frac{1}{0}:=\infty$, $\frac{1}{\infty}:=0$ and $\ln 0:=-\infty$. Because of this last part, $p_n^0=0$ for some $n$ implies $a_0^+=a_0^-=\infty$.\par The intuition behind $a_0^+$ is that this quantity measures how much $P^+$ and $Q^+$ differ: If $a_0^+>1$, then $P^+=Q^+$ holds, while $a_0^+<1$ implies $Q^+(\{0\})=1$. 
Similarly to the quantity $D_k(h)$ in lemma \ref{lemma1}, we write
\[D:= \frac{\min\{1,a^-\}}{\varepsilon^+}+\frac{\min\{1,a^+\}}{\varepsilon^-}\]
so that 
\[\inf\left\{C(b_1,b_2,h):b_1\geq \frac{h}{\varepsilon^-},b_2\geq \frac{h}{\varepsilon^+}\right\} =h D\]
Then we have the following result
\begin{theorem}\label{thm2}
We work under the assumptions of lemma \ref{lemma2}. Assume that for some $\kappa> 0$
\begin{equation}\label{eq:thm2vor1}
\lim_{n\to\infty}\frac{\ln\min\{p_n^+,p_n^-\}}{-\ln^{\kappa} n} = c\in (0,\infty)
\end{equation}
then 
\begin{equation}\label{eq:thm2res1}
\frac{1}{1+\kappa}\left(\frac{\kappa}{1+\kappa}\right)^\kappa D\leq\liminf_{n\to\infty}\frac{\ln \mathbb{P}(\tau > n)}{-c\ln^{\kappa+1} n} \leq \limsup_{n\to\infty}\frac{\ln \mathbb{P}(\tau > n)}{-c\ln^{\kappa+1} n} \leq D
\end{equation}
\end{theorem}
\subsection{The stretched-exponential case}
We have obtained some results for the case where $\min\{p_n^+,p_n^-\}$ is of constant order or of the order $e^{-c\ln^{\kappa} n}$ for some $\kappa,c>0$. Now we give a weaker result the remaining case.
\begin{theorem}\label{thm3}Assume that 
the assumptions of lemma \ref{lemma2} hold, but instead of (\ref{eq:thm2vor1}) we have some $\kappa>0$ such that
\begin{equation}\label{eq:thm3vor2}
\lim_{n\to\infty}\frac{\ln(\min\{p_n^+,p_n^-\})}{-n^\kappa} =c\in (0,\infty)
\end{equation}
Then we have 
\begin{equation}\label{eq:thm3res2}
\frac{\kappa}{1+5\kappa}\leq\liminf_{n\to\infty}\frac{\ln (-\ln \mathbb{P}(\tau > n))}{\ln n} \leq  \limsup_{n\to\infty}\frac{\ln (-\ln \mathbb{P}(\tau > n))}{\ln n} \leq  \kappa
\end{equation}
\end{theorem}
Let us briefly consider why the concept of valleys is not well suited for handling the case of Theorem \ref{thm3}. Consider two measures $P$ and $\tilde{P}$ on the environment such that both satisfy 
\[\min\{p_n^+,p_n^-\}\sim ce^{-n^{\kappa}}\]
but 
\begin{equation}\label{eq:becher}P\left(\omega_0^0=0,\omega_0^+=\omega_0^-=\frac{1}{2}\right)=:\gamma>0\end{equation}
and
\begin{equation}\tag{\ref{eq:becher}'}\tilde{P}\left(\omega_0^0=0,\omega_0^+=\omega_0^-=\frac{1}{2}\right)=0\end{equation}
Denote the RWREs having measure $P$ and $\tilde{P}$ for the environment by $\mathbb{P}$ respectively $\tilde{\mathbb{P}}$. We define for the interval $I_n:=[\;-n^{\frac{1}{3}},n^{\frac{1}{3}}\;]$ the event
\[S:=\left\{\forall x\in I_n: \omega_x^0=0,\omega_x^+=\omega_x^-=\frac{1}{2}\right\}\]
that the RWRE inside the interval $I_n$ is a simple random walk. We have $P(S)\sim \exp(2n^{\frac{1}{3}}\ln \gamma)$. We can use the following result for simple random walks starting in zero: 
\begin{theorem}
Let $U_n$ be the first time that the random walk hits the boundary an interval of length $l(n)$ around zero. Moreover assume 
\[\lim_{n\to\infty} l(n)=\infty\quad\text{and}\quad \lim_{n\to\infty}\frac{l^2(n)}{n}=0\]
Then for $c>0$
\[\lim_{n\to\infty} \frac{l^2(n)}{n}\ln P(U_n\geq cn)=-\frac{c\pi^2}{8}\]
\end{theorem}
The proof can be found in \cite{srw}, pp 237. Using this Theorem with $l(n)=n^\frac 13$, we have
\[\mathbb{P}(\tau>n)\geq P(S)\inf_{\omega\in S} P_\omega(U_n\geq n)\geq\exp\left\{\left(-\frac{\pi}{8}+2\ln\gamma-\epsilon\right)n^{\frac{1}{3}}\right\} \]
On the other hand there is no corresponding lower bound for $\tilde{\mathbb{P}}$. Now for $\kappa>\frac{1}{3}$ this bound is better than the one obtained by surviving on the inside of a valley, since the cost of creating a valley is of the order $e^{-cn^{\kappa}\ln (n)}$. However the probability that a valley occurs does not depend on whether (\ref{eq:becher}) or (\ref{eq:becher}') hold.
\subsection{Remarks}
\begin{remark}
In the definition of the model, we introduced the quantity $r$ as the probability that the random walk dies once it stays at the same location. One notes that $r$ does not appear in the constants in (\ref{eq:lemma1res1}) and (\ref{eq:lemma2res2}). That is because we have obtained the results under conditions (\ref{eq:lemma1vor1}) and (\ref{eq:lemma2vor2}), which imply that there are valleys, where on the inside the probability for survival is 1 in the first case, and some positive constant depending on $r$ in the second case. In the proofs we show that the survival probability is dominated by the probability that such a valley is formed, which depends only on the environment and not on $r$.
\end{remark}
\begin{remark}
An interesting question is whether for any sequence $(q_n)_{n\in\mathbb{N}}$, we can find a probability measure for which $\min\{p_n^+,p_n^-\}$ decays exactly as $(q_n)_n$. Indeed we can easily construct a suitable measure. It is enough to describe the distribution of $\omega_0$ because $P$ is a product measure. For this, let $(q_n)_n$ be any real sequence in $[0,1]$ decreasing to zero. Define 
\[\Pi^+_n := \left(\frac{1+\varepsilon}{2+\varepsilon} \left(1-\frac{1}{n}\right),\frac{1}{n},\frac{1}{2+\varepsilon} \left(1-\frac{1}{n}\right)\right)\]
\[\Pi^-_n := \left(\frac{1}{2+\varepsilon} \left(1-\frac{1}{n}\right),\frac{1}{n},\frac{1+\varepsilon}{2+\varepsilon} \left(1-\frac{1}{n}\right)\right)\]
Then for any $n$
\[\omega_0 =\Pi^+_n\implies\rho_0=1+\varepsilon\]\[\omega_0 =\Pi^-_n\implies\rho_0 = (1+\varepsilon)^{-1}\]
Now we define $P$ as the discrete probability measure taking values in the set $\{\Pi_n^\pm:n\in\N\}$, with
\[P(\omega_0 =\Pi_n^+ )=P(\omega_0 =\Pi_n^- )=\frac{q_{n}-q_{n+1}}{2c}\in [0,1], n\in\mathbb{N}\]
Here $c:=q_0$ is the normalizing constant such that $P$ is a probability measure. Now independently of $n$
\[P\left(\ln\rho_0\in\cdot\left|\omega_0^0\leq\frac{1}{n},\rho_0\neq 1\right.\right)=\frac 12 \delta_{-\ln(1+\varepsilon)}(\cdot)+\frac 12 \delta_{\ln(1+\varepsilon)}(\cdot)\]
where $\delta_x$ is the Dirac distribution in $x$. Conditions (\ref{eq:convdistrp}),(\ref{eq:convdistrm}) are satisfied with 
\[P^+([t,\infty))=P^-([t,\infty))=\mathbbm{1}_{[\ln(1+\varepsilon),\infty)}(t)\]
Moreover we get $\varepsilon^+=\varepsilon^-=\delta^+=\delta^-=\ln(1+\varepsilon)$ and $a^+=a^-=1$. We summarize this in the following corollary:
\begin{corollary}For every sequence $(q_n)_n$ in $[0,1]$ decreasing to zero there is a probability measure $P$ such that 
\[\min\{p_n^+,p_n^-\}=q_n \ \forall n\in\mathbb{N}\]
In particular, for every $\kappa>0$ we find a probability measure $P$ and constants $c_1,c_2>0$ such that for all $n$ large enough
\[e^{-c_1\ln^{1+\kappa} n}\leq \mathbb{P}(\tau>n)\leq e^{-c_2\ln^{1+\kappa} n}\]
\end{corollary}
\end{remark}
\section{Proofs - the polynomial case}
The proofs in this section modify the proof of Theorem 1.3 in \cite{soft}\footnote{\;cf. \cite{soft}, pp 7}. 
We use Cramer's Theorem from large deviations\footnote{\;cf. for example section 2.2.1 in \cite{largedev}, in particular corollary 2.2.19}:
\begin{theorem}\label{cramer}
Let $(X_n)_{n\geq 1}$ be  i.i.d. random variables taking values in $\mathbb{R}$ with law $Q$ such that the moment generating function $\Lambda(t):= E_Q(e^{tX_1})$ is finite for some $t>0$. Then for all $x>E(X_1)$ we have
\begin{equation}\label{eq:cramer}\lim_{k\to\infty}-\frac{\ln Q(\frac{1}{k}\sum_{i=1}^k X_i\geq x)}{k}=\Lambda^*(x)\end{equation}
where $\Lambda^*$ is the Legendre transform of $Q$:
\[\Lambda^*(x)=\sup_{t\geq 0}\{tx-\ln\Lambda(t)\}\]
\end{theorem}
\begin{proof}[Proof of lemma \ref{lemma1}]
For ease of notation we will omit the integer parts, that is we treat $b_2\ln n$, $b_1\ln n$ and $h\ln n$ as integers. We fix $k$ and say that a location $x$ is \textbf{safe} if $\omega_x^0\leq\frac{1}{k}$. Consider the following events:
\begin{subequations}\begin{align}
A_n^+:=&\{\omega_i\text{ is safe }\forall i \in [x,x+b_2\ln n]\}\label{eq:defevents1}\\
A_n^-:=&\{\omega_i\text{ is safe }\forall i \in [x-b_1\ln n,x]\}\\
B_n^+:=&\left\{\sum_{i=x}^{x+b_2\ln n} \ln\rho_i\geq h\ln n\right\}\\
B_n^-:=&\left\{\sum_{i=x-b_1\ln n}^{x} \ln\rho_i\geq h\ln n\right\}\label{eq:defevents2}\\
C_n^+:=&\{V(x)=\min\{V(y):y\in [x,x+b_2\ln n]\}\}\\
C_n^-:=&\{V(x)=\min\{V(y):y\in [x-b_1\ln n,x]\}\}\end{align}\end{subequations}
Now $S_n^k$ decomposes as
\[S_n^k(x,b_1,b_2,h)=A_n^+\cap B_n^+\cap C_n^+\cap A_n^-\cap B_n^-\cap C_n^-\]
Conditioned on $A_n^+$, the distribution of $\ln\rho_i$ for $i=0,...,b_2\ln n$ is
\[Q(\cdot):= P\left(\ln\rho_0\in\cdot\left|\omega_0^0\leq\frac{1}{k}\right.\right)\]
For $b_2<h(E(\ln\rho_0|\omega_0^0\leq\frac{1}{k}))^{-1}$ we can applying Theorem \ref{cramer} with $k$ replaced by $b_2\ln n$ and $x$ by $\frac{h}{b_2}$, which yields
\begin{equation}\label{eq:label1}
\lim_{n\to\infty}\frac{\ln P(B_n^+|A_n^+)}{-\ln n} = b_2\Lambda^*\left(\frac{h}{b_2}\right)
\end{equation}
where 
\begin{equation}\label{eq:label2}
\Lambda^*(x)=\sup_{t\geq 0}\left \{tx-\ln E\left(\rho_0^t\left|\omega_0^0\leq\frac{1}{k}\right.\right)\right\}
\end{equation}
For larger values of $b_2$ the limit in (\ref{eq:label1}) is zero, while the supremum in (\ref{eq:label2}) is attained at $t=0$ and also equals zero. The following relation is therefore valid for all $b_1, b_2>0$
\[\lim_{n\to\infty}\frac{\ln[ P(B_n^+|A_n^+)P(A_n^+)P(B_n^-|A_n^-)P(A_n^-)]}{-\ln n}=-(b_2+b_1)\ln P\left(\omega_0^0\leq\frac{1}{k}\right)+\]
\[+\sup_{t\geq 0}\left \{th-b_2\ln E\left(\left.\rho_0^t\right|\omega_0^0\leq\frac{1}{k}\right)\right\}+\sup_{t\geq 0}\left \{th-b_1\ln E\left(\left.\rho_0^{-t}\right|\omega_0^0\leq\frac{1}{k}\right)\right\}\]
Now we have (using $h>0$)
\[b_2\ln n P\left(A_n^+,B_n^+,C_n^+\right)=\]
\[=\sum_{j=0}^{b_2\ln n} P\left(A_n^+,\sum_{i=0}^{b_2\ln n}\ln\rho_i\geq h\ln n,V(0)=\min\Big\{V(x):x=0,...,b_2\ln n\Big\}\right)\]
\[\geq\sum_{j=0}^{b_2\ln n} P\left(A_n^+,\sum_{i=0}^{b_2\ln n}\ln\rho_i\geq h\ln n, V(j)=\min\Big\{V(x):x=0,...,b_2\ln n\Big\}\right)\]
\[\geq P\left(A_n^+,\sum_{i=0}^{b_2\ln n}\ln\rho_i\geq h\ln n , \exists j : V(j)=\min\Big\{V(x):x=0,...,b_2\ln n\Big\}\right)\]
\[=P\left(A_n^+,B_n^+\right)\]
Therefore 
\[\lim_{n\to\infty}\frac{\ln P(A_n^+, B_n^+, C_n^+)}{\ln n}\geq \lim_{n\to\infty}\frac{\ln P(A_n^+, B_n^+)+\ln (b_2\ln n)}{\ln n}=\]
\[=\lim_{n\to\infty}\frac{\ln P(A_n^+, B_n^+)}{\ln n}\]
Since the converse inequality is always true, this shows (\ref{eq:lemma1res1}). By definition
\[D_k(h)=\inf_{b_2>0}\left\{\sup_{t\geq 0}\left\{ht-b_2\ln E\left(\rho_0^t\left|\omega_0^0 \leq\frac{1}{k}\right.\right)\right\} - b_2\ln P\left(\omega_0^0\leq\frac{1}{k}\right)\right\}+\]
\begin{equation}\label{eq:respres}+\inf_{b_1>0}\left\{\sup_{t\geq 0}\left\{ht-b_1\ln E\left(\rho_0^{-t}\left|\omega_0^0\leq \frac{1}{k}\right.\right)\right\} - b_1\ln P\left(\omega_0^0\leq\frac{1}{k}\right)\right\}\end{equation}
For the second claim, consider the functions 
\[f_k^\pm:[0,\infty)\to\mathbb{R},t\mapsto E\left(\rho_0^{\pm t}\left|\omega_0^0\leq\frac{1}{k}\right.\right)\]
Because of uniform ellipticity both functions are finite everywhere, and additionally they are strictly convex, infinitely differentiable and satisfy $f^\pm(0)=1$ and $f^\pm(t)\to\infty$ for $t\to\infty$. Therefore there exist unique $t_k^\pm>0$ such that
\begin{equation}\label{eq:deft_k}
f_k^+(t_k^+)=E\left(\rho_0^{t_k^+}\left|\omega_0^0\leq\frac{1}{k}\right.\right)=\frac{1}{P(\omega_0^0\leq\frac{1}{k})}>1
\end{equation}
\[f^-_k(t_k^-)=E\left(\rho_0^{-t_k^-}\left|\omega_0^0\leq\frac{1}{k}\right.\right)=\frac{1}{P(\omega_0^0\leq\frac{1}{k})}>1\]
Inserting $t_k^+$ and $t_k^-$ in (\ref{eq:respres}) yields
\begin{equation}\label{eq:respres2}D_k(h)\geq h(t_k^++t_k^-)\end{equation}
We claim that equality holds in (\ref{eq:respres2}). To see this, set
\[g_b(t):=th- b\ln \left[E\left(\rho_0^t\left|\omega_0^0\leq\frac{1}{k}\right.\right)P\left(\omega_0^0\leq\frac{1}{k}\right)\right]\]
The function $g_b$ is concave for every $b>0$, since by Cauchy-Schwarz:
\[E\left(\left.\rho_0^{\frac{t_1+t_2}{2}}\right|\omega_0^0\leq\frac{1}{k}\right)^2\leq E\left(\rho_0^{t_1}\left|\omega_0^0\leq\frac{1}{k}\right.\right)E\left(\rho_0^{t_2}\left|\omega_0^0\leq\frac{1}{k}\right.\right)\]
By dominated convergence (using again that $P$ is uniformly elliptic) we get $\frac{d}{dt} E(\rho^t\omega_0^0\leq\frac{1}{k})=E(\frac{d}{dt}\rho^t|\omega_0^0\leq\frac{1}{k})$ and therefore
\[g_b'(t)=h-\frac{bE(\rho_0^t\ln\rho_0|\omega_0^0\leq\frac{1}{k})}{E(\rho_0^t|\omega_0^0\leq\frac{1}{k})}\]
and setting 
\[\bar{b}=\bar{b}(t_k^+):=h \left(E\left(\rho_0^{t_k^+}\ln\rho_0\left|\omega_0^0\leq\frac{1}{k}\right.\right)P\left(\omega_0^0\leq\frac{1}{k}\right)\right)^{-1}\]
yields that $g_{\bar{b}}'(t_k^+)=0$ and thus by concavity
\[\inf_{b_2>0}\{\sup_{t\geq 0}\{g_b(t)\}\}\leq g_{\bar{b}}(t_k^+)=ht_k^+\]
Using the same reasoning for the second line in (\ref{eq:respres}) we conclude
\[D_k(h)=h(t_k^++t_k^-)\]
It remains to show $\lim_{k\to\infty}t_k^+ = t_\infty^+$. First note that $\{t_k^+:k\in\N\}$ is bounded: Because of (\ref{eq:lemma1vor1}) we find $\alpha>0$ such that 
\[P\left(\omega_0^0=0,\omega_0^-\geq\frac{1}{2}+\alpha \right)=\beta>0\]
Therefore
\[\frac{1}{P(\omega_0^0=0)}\geq\frac{1}{P(\omega_0^0\leq\frac{1}{k})}=E\left(\rho_0^{t_k^+}\left|\omega_0^0\leq\frac{1}{k}\right.\right)\geq E\left(\rho_0^{t_k^+}\left|\omega_0^0=0\right.\right)P(\omega_0^0=0)\geq\]
\[\geq \left(\frac{1+2\alpha}{1-2\alpha}\right)^{t_k^+}\beta P(\omega_0^0=0)\]
That is: 
\[t_k^+\leq \left(-2\ln P(\omega_0^0=0)-\ln\beta\right)\left(\ln\frac{1+2\alpha}{1-2\alpha}\right)^{-1}\]
Now choose any subsequence such that $t':=\lim_{i\to\infty} t_{k_i}$ exists. Then
\[\frac{1}{P(\omega_0^0=0)}=\lim_{i\to\infty}\frac{1}{P(\omega_0^0\leq\frac{1}{k_i})}=\lim_{i\to\infty}E\Big(\rho_0^{t_{k_i}^+}\Big|\omega_0^0\leq\frac{1}{k_i}\Big)\]
\begin{equation}\label{eq:flasche}=\lim_{i\to\infty}\frac{E\Big(\rho_0^{t_{k_i}^+}\mathbbm{1}_{\left\{\omega_0^0\leq\frac{1}{k_i}\right\}}\Big)}{P(\omega_0^0\leq\frac{1}{k_i})}=\frac{E\left(\rho_0^{t'}\mathbbm{1}_{\left\{\omega_0^0=0\right\}}\right)}{P(\omega_0^0=0)}=E(\rho_0^{t'}|\omega_0^0=0)\end{equation}
Here the second to last equality is due to dominated convergence, the boundedness of $t_n^+$ and uniform ellipticity. We conclude that $t'$ satisfies (\ref{eq:deft_k}). Since $t_\infty^+$ is the unique positive value satisfying this equation, we see $t'=t_\infty^+$.
\end{proof}
\begin{proof}[Proof of Theorem \ref{thm1}]
We start with the lower bound. Let $\alpha>0$ and choose $b_1,b_2>0$ arbitrary. We write
\[U_n:=\inf\Big\{n\geq 0\Big|X_n \in \{\lfloor -b_1\ln n\rfloor,\lfloor b_2\ln n\rfloor\}\Big\}\]
Then
\[\mathbb{P}^0(\tau>n)\geq P(S^\infty_n(0,b_2,b_1,1))\inf_{\omega\in S^\infty_n}\Big\{P_{\omega}(U_n> n )P_\omega(\tau>n|U_n>n)\Big\}\]
By lemma \ref{lemma1} there is an $n_0$ so that for all $n\geq n_0$ we have:
\[P(S^\infty_n(0,b_2,b_1,1))\geq n^{-C_\infty(b_2,b_1,1)-\alpha}\]
Remember that for $\omega\in S^\infty_n$, zero has minimal potential in $[-b_1\ln n,b_2\ln n]$ and $H=H(-b_1\ln n,b_2\ln n )\geq \ln n$. By corollary \ref{korrpot} we get for some $c_2>0$:
\[P_{{\omega}}^0({U_n}\geq n)= P_{{\omega}}^0\left(\gamma_3\ln(2\ln (n) (b_2+b_1))\frac{{U_n}}{e^H}\geq \gamma_3\ln(2\ln (n) (b_2+b_1))\frac{ n}{e^H} \right)\geq \] 
\begin{equation}\label{eq:applycor1}\geq\frac{1}{2\ln(n)(b_2+b_1)}\exp\left(-\gamma_3\ln(2\ln (n) (b_2+b_1))\frac{ n}{e^H}\right)\geq\exp(-c_2 \ln \ln n)\end{equation}
Moreover, conditional on the random walk never leaving the interval \newline$[-b_1\ln n,b_2\ln n]$, the probability of surviving is $1$, since here the probability for staying at the same location equals zero. In total we get
\[\liminf_{n\to\infty}\frac{\ln \mathbb{P}^0(\tau>n)}{\ln n}\geq -C_\infty(b_1,b_2,1)-\alpha\]
Letting $\alpha$ tend to zero and taking the infimum over all $b_1,b_2$ yields 
\begin{equation}\label{eq:lowerbound}
\limsup_{n\to\infty}-\frac{\ln \mathbb{P}^0(\tau>n)}{\ln n}\leq D_\infty(1)
\end{equation}
Now we address the upper bound: Let $\alpha,\beta,\gamma,\delta>0$ and $k\in\mathbb{N}$. We say $x$ is \textbf{dangerous} if $\omega_x^0>\frac{1}{k}$ and we write $\Theta$ for the set of dangerous locations. An environment is called \textbf{good} if the following conditions apply:
\begin{equation}\label{eq:goodenv1}
\forall x \in[-\ln^{1+\alpha} n,\ln^{1+\alpha} n],b_2,b_1>0\text{ we do \underline{not} have }S_n^k(x,b_2,b_1,1-\delta)
\end{equation}
\begin{equation}\label{eq:goodenv2}
|\Theta \cap [0,\ln^{1+\alpha} n]|, |\Theta\cap[-\ln^{1+\alpha} n,0]|\geq\frac{1}{2} P(\omega_0^0>\frac{1}{k})\ln^{1+\alpha} n
\end{equation}
We estimate the probability of not being in a good environment: In (\ref{eq:goodenv2}), we have the event that a binomial random variable with $\lfloor \ln^{1+\alpha}n\rfloor$ trials and success probability $P\left(\omega_0^0\leq\frac{1}{k}\right)$ has at least $\frac{1}{2}P(\omega_0^0\leq\frac{1}{k})\ln^{1+\alpha} n$ successes. The probability that this does not occur can be bounded by $e^{-c_3\ln^{1+\alpha} n}$ for some $c_3>0$. Now concerning (\ref{eq:goodenv1}): Because of uniform ellipticity a valley of depth $1-\delta$ requires
\[{b_2},{b_1}\geq (1-\delta)\left(\ln\frac{1-\varepsilon_0}{\varepsilon_0}\right)^{-1}\]
So the number of valleys in the interval $[-\ln^{1+\alpha} n,\ln^{1+\alpha} n]$ is at most 
\[\frac{2\ln\frac{1-\varepsilon_0}{\varepsilon_0}\ln^{1+\alpha} n}{2( 1-\delta)\ln n}=:c_4 \ln^\alpha n\]
Since those valleys must remain inside $[-\ln^{1+\alpha} n,\ln^{1+\alpha} n]$ and may only have integer length we can estimate for large enough $n$
\[P(\exists b_2,b_1>0:S_n^k(0,b_2,b_1,1-\delta))\leq \sum_{b_1,b_2=1}^{\lfloor \ln^{1+\alpha} n\rfloor} P\left(S_n^k\left(0,\frac{b_2}{\ln n},\frac{b_1}{\ln n},1-\delta\right)\right)\leq\]
\[\leq \sum_{b_1,b_2=1}^{\lfloor \ln^{1+\alpha} n\rfloor}n^{-C_k\left(\frac{b_2}{\ln n},\frac{b_1}{\ln n},1-\delta\right)+\beta}\leq \ln^{2(1+\alpha)} (n) n^{-D_k(1-\delta)+\beta}\]
So for large $n$ we have
\begin{equation}\label{eq:nichtgut} P(\omega \text{ is not good })\leq 2 e^{-c_3\ln^{1+\alpha}n}+ c_4 e^{(-D_k(1-\delta)+\beta)\ln n + \ln (2\alpha(1-\alpha)\ln^2 n)}\end{equation}
From here on, let $\omega$ be a good environment. The idea is to consider time intervals of length $n^{1-\frac{\delta}{2}}$ and show that in each such interval the probability of hitting a dangerous locations is at least a constant. Since there are $\lfloor n^{\frac{\delta}{2}}\rfloor$ such time intervals we can then conclude that the extinction probability decays fast.
\par Let $a,b\in \Theta$ be such that $[a,b]\subset[-\ln^{1+\alpha}n,\ln^{1+\alpha}n]$ and $(a,b)\cap\Theta=\emptyset$. Then by (\ref{eq:goodenv2}) the interval $[a,b]$ satisfies $H:=H(a+1,b-1)\leq (1-\delta)\ln n$. However we can not directly use lemma \ref{lemmapot1} since it applies only to environments with no holding times. Therefore consider a random walk $\overline{X_n}$ according to the measure $P_{\overline{\omega}}$, which is $P_\omega$ conditioned on the random walk never staying in one place:
\begin{equation}\label{eq:baromega}\overline{\omega}_x:= \left(\frac{\omega_x^-}{\omega_x^++\omega_x^-},0, \frac{\omega_x^+}{\omega_x^++\omega_x^-}\right)\text{ for all }x\in\mathbb{Z}\end{equation}
Obviously this does not change the potential. We write $U_{a,b}$ (resp. $\overline{U}_{a,b}$) for the first time $X_n$ (resp. $\overline{X}_n$) hits $\{a,b\}$ and we choose $\gamma>1$ small enough that $(1-\frac{\gamma}{k})>0$. Using lemma \ref{lemmapot1} yields:
\[\min_{x\in(a,b)}P_{\overline{\omega}}^x\left((\overline{X}_t)\notin (a,b) \text{ for some }t\leq \left(1-\frac{\gamma}{k}\right) n^{1-\frac{\delta}{2}}\right)=\]
\[=1-\max_{x\in(a,b)}P_{\overline{\omega}}^x\left(\frac{\overline{U}_{a,b}}{\gamma_1 (c-a)^4 e^H}> \frac{ (1-\frac{\gamma}{k})n^{1-\frac{\delta}{2}}}{\gamma_1 (c-a)^4 e^H}\right)\geq\]
\[
\geq1- \exp\left(-\frac{ (1-\frac{\gamma}{k})n^{1-\frac{\delta}{2}}}{\gamma_1 (c-a)^4 e^H}\right)\geq1-\exp\left(-\frac{ (1-\frac{\gamma}{k})n^\frac{\delta}{2}}{16\gamma_1\ln^{4+4\alpha} n}\right)
\]
We choose $n$ large, so that this probability is at least $\frac{1}{2}$. Let $Y$ be the number of time steps until $U_{a,b}\land n^{1-\frac{\delta}{2}}$ during which $(X_n)$ remains at the same location. Since $\omega_\cdot^0\leq\frac{1}{k}$ for all locations inside the valley, $Y$ is dominated  by a binomial random variable with $\lfloor n^{1-\frac{\delta}{2}}\rfloor$ trials and success probability $\frac{1}{k}$. Let $E_n^1:=\{Y\leq\frac{\gamma}{k} n^{1-\frac{\delta}{2}}\}$ and note that $P(E_n^1)\to 1$. Under both $E_n^1$ and $\left\{\overline{U}_{a,b}\leq (1-\frac{\gamma}{k}) n^{1-\frac{\delta}{2}}\right\}$ we see 
\[U_{a,b}=\overline{U}_{a,b}+Y\leq\left(1-\frac{\gamma}{k}\right)n^{1-\frac{\delta}{2}}+\frac{\gamma}{k} n^{1-\frac{\delta}{2}}=n^{1-\frac{\delta}{2}}\]
Choosing $n$ large enough that $P(E_n^1)\geq\frac{1}{2}$, we conclude
\begin{equation}\label{eq:trapprob}
\min_{x\in(a,b)}P_{{\omega}}^x\left(({X}_t)\notin (a,b) \text{ for some }t\leq n^{1-\frac{\delta}{2}}\right)\geq \frac{1}{4}
\end{equation}
Now we consider the event 
\[E_n^2:=\{(X_t)\notin[-\ln^{1+\alpha}n,\ln^{1+\alpha}n] \text{ for some }t\leq n\}\]
Since $\omega$ is good, the random walk has to pass at least 
\begin{equation}\label{eq:refrefref}
\frac{1}{2} P\left(\omega_0^0>\frac{1}{k}\right)\ln^{1+\alpha}(n)=:c_5\ln^{1+\alpha}(n)
\end{equation}
dangerous locations. However, on $\left(E_n^2\right)^c$, let $Z$ be the number of time intervals of the form 
\[\left[kn^{1-\frac{\delta}{2}},(k+1)n^{1-\frac{\delta}{2}}\right),\quad k=0,...,\lfloor n^{\frac{\delta}{2}}\rfloor\]
where the random walk hits a dangerous location. By (\ref{eq:trapprob}), $Z$ is dominated by a binomial random variable with $\lfloor n^{\frac{\delta}{2}}\rfloor$ trials and success probability $\frac{1}{4}$. We write $E_n^3$ for the event $\left\{Z\leq \frac{1}{8}n^\frac{\delta}{2}\right\}$, and note that for some $c_6>0$, 
\begin{equation}\label{eq:refrefrefref}P_\omega\left(E_n^3\right)\leq \exp(-c_6n^{\frac{\delta}{2}})\end{equation}
Finally we get
\[\mathbb{P}^0(\tau>n)\leq P(\omega\text{ is not good})+\sup_{\omega\text{ is good}}\Big\{ P_\omega(\tau>n)\Big\}\leq\]
\[ P(\omega\text{ is not good})+\sup_{\omega\text{ is good}} \Big\{P_\omega\left(E_n^3\right)+P_\omega(\tau>n|E_n^2)+ P_\omega\left(\tau>n\left|(E_n^2)^c, (E_n^3)^c\right.\right)\Big\}\]
\[\leq e^{-c_5\ln^{1+\alpha}n}+ \exp\left[ (-D_k(1-\delta)+\beta)\ln n + \ln (2\alpha(1-\alpha) \ln n)\right]+\]
\[+\exp\left(-c_6n^{\frac{\delta}{2}}\right)+\exp\left(c_5\ln \left(1-\frac{r}{k}\right) \ln^{1+\alpha} n\right)+\exp\left(\frac{1}{8}\ln \left(1-\frac{r}{k}\right)n^{\frac{\delta}{2}}\right)\]
In the last line, the first three terms are due to (\ref{eq:refrefref}), (\ref{eq:nichtgut}) and (\ref{eq:refrefrefref}). The last two terms come from the fact that at every visit to a dangerous location the process dies with probability at least $1-\frac{r}{k}$, and on $E_n^2$ and $(E_n^3)^c$ we have a lower bound on the number such visits. Letting $\beta$ tend to zero this shows	
\begin{equation}\label{eq:upperbound}
\limsup_{n\to\infty}\frac{\ln \mathbb{P}^0(\tau>n)}{\ln n}\leq -D_k(1-\delta)\xrightarrow{k\to \infty}-D_\infty(1-\delta)\xrightarrow{\delta\to 0}-D_\infty(1)
\end{equation}
The claim follows by (\ref{eq:lowerbound}) and (\ref{eq:upperbound}).
\end{proof}
\section{Proofs - the intermediate case}
\begin{proof}[Proof of lemma \ref{lemma2}]
For $n$ fixed we will call a location $x$ \textbf{safe} if $\omega_x^0\leq\frac{1}{n}$, and we say that a safe location $x$ is \textbf{positive}, \textbf{neutral} or \textbf{negative} if $\ln\rho_x$ is positive, zero or negative, respectively. We start with the upper bound. Recalling definitions (\ref{eq:defevents1})-(\ref{eq:defevents2}) in the previous proof we have
\[P(T_n(x,b_1,b_2,h))\leq P(A_n^+,B_n^+)P(A_n^-,B_n^-)\]
We will do the calculations for the left probability. By the definitions of $\varepsilon_n^+$ and $\delta_n^-$ we know that, on $A_n^+$, a positive location in $[x,x+b_2\ln n]$ satisfies $\ln\rho_x\leq\varepsilon_n^+$ whereas a negative/neutral location satisfies $-\ln\rho_x\geq \delta_n^-$. Let $k$ be the number of positive locations in $[x,x+b_2\ln n]$. Then
\[\sum_{i=x}^{x+\lfloor b_2\ln n\rfloor}\ln\rho_i\leq k\varepsilon_n^+-\delta_n^-(\lfloor b_2\ln n\rfloor-k)\]
The event $B_n^+$ therefore requires $k\geq \lceil d_n\ln n\rceil$, where $d_n:=\frac{h+\delta_n^- b_2}{\varepsilon_n^++\delta_n^-}$. We write $B_{n,k}$ for the event that exactly $k$ locations are positive and $(\lfloor b_2\ln n\rfloor-k)$ are negative or neutral. Since $A_n^+\cap B_n^+ = \bigcup_{k=\lceil d_n\ln n\rceil}^{\lfloor b_2\ln n\rfloor} B_{n,k}$ we get
\[P(A_n^+\cap B_n^+)=\sum_{k=\lceil d_n\ln n\rceil}^{\lfloor b_2\ln n\rfloor} P(B_{n,k})\leq (b_2-d_n)\ln n \max_{k} P(B_{n,k})\]
\[=(b_2-d_n)\ln n \max_k \binom{\lfloor b_2\ln n\rfloor}{k} \left(p_n^+\right)^{k}(p_n^-+p_n^0)^{\lfloor b_2\ln n\rfloor-k} \]
Here and from now on the maximum is taken over $k=\lceil d_n\ln n\rceil,...,\lfloor b_2\ln n\rfloor$. Using $\binom{n}{k}\leq \left(\frac{ne}{k}\right)^k$ we see that for some $c_7>0$:
\begin{equation}\label{eq:expfct}P(A_n^+\cap B_n^+)\leq (c_7)^{\ln n} \max_ke^{k(\ln p_n^+-\ln (p_n^-+p_n^0))+\lfloor b_2\ln n\rfloor\ln (p_n^-+p_n^0)}\end{equation}
We now use the fact that for decreasing sequences $(x_n)_{n\in\N},(y_n)_{n\in\N},(z_n)_{n\in\N}$ in $(0,1)$ we have 
\[\lim_{n\to\infty}\frac{\ln y_n}{\ln x_n}=y,\lim_{n\to\infty}\frac{\ln z_n}{\ln x_n}=z\implies\lim_{n\to\infty}\frac{\ln (y_n+z_n)}{\ln x_n}=\min\{y,z\}\]
Assume first that $\min\{a^+,a^+_0\}< 1$. Then for all $\eta\in(0,1-\min\{a^+,a^+_0\})$ we can choose $n$ large enough that
\[(\min\{a^++a^+_0\}-\eta)\ln p_n^+\geq\ln (p_n^-+p_n^0)\geq (\min\{a^++a^+_0\}+\eta)\ln p_n^+\]
which implies
\[\ln p_n^+-\ln (p_n^-+p_n^0)\leq \ln p_n^+(\underbrace{1-\min\{a^+,a^+_0\}-\eta}_{>0})\]
Therefore the maximum from equation (\ref{eq:expfct}) is attained at $k$ equal to $\lceil d_n\ln n\rceil$, and inserting this yields
\[P(A_n^+\cap B_n^+)\leq (c_7)^{\ln n} e^{\ln n\ln p_n^+(d_n (1-\min\{a^+,a^+_0\}-\eta)+b_2(\min\{a^+,a^+_0\}+\eta))}\]
Since $\ln p_n^+$ tends to $-\infty$ this implies
\[\liminf_{n\to\infty}\frac{\ln P(A_n^+\cap B_n^+)}{\ln n\ln p_n^+}\geq \lim_{n\to\infty}d_n (1-\min\{a^+,a^+_0\}-\eta)+b_2(\min\{a^+,a^+_0\}+\eta))\]
\[=\frac{h+\delta^- b_2}{\varepsilon^++\delta^-}(1-\min\{a^+,a^+_0\}-\eta)+b_2(\min\{a^+,a^+_0\}+\eta)\]
\[\xrightarrow{\eta\to 0}\frac{h+b_2\delta^-+\min\{a^+,a_0^+\} (-h+b_2\varepsilon^+)}{\varepsilon^++\delta^-}=C^+(h,b_2)\]
Now let $\min\{a^+,a^+_0\}\geq 1$, that is we get $C^+(h,b_2)=b_2$. Then for $\eta>0$ we have $1-\min\{a^+,a^-\}-\eta<0$ and (for $n$ large enough) the maximum in equation (\ref{eq:expfct}) is attained at $k$ equal to $\lfloor b_2\ln n\rfloor$. That is,
\[\liminf_{n\to\infty}\frac{\ln P(A_n^+\cap B_n^+)}{\ln n\ln p_n^+}\geq b_2= C^+(b_2,h)\]
Now we would like a corresponding lower bound on the same probability. The same reasoning as in the proof of lemma \ref{lemma1} shows that 
\[\lim_{n\to\infty}\frac{\ln P(A_n^+\cap B_n^+\cap C_n^+)}{\ln n\ln p_n^+}= \lim_{n\to\infty}\frac{\ln P(A_n^+\cap B_n^+)}{\ln n\ln p_n^+}\]
Let $\eta\in (0,\varepsilon^+)$ and choose
\[d:=\left\{\begin{matrix}\frac{h+b_2(\delta^-+\eta)}{\varepsilon^++\delta^-} & \text{ if }  \min\{a^+,a^+_0\}<1 \\[2mm] b_2 & \text{else}\end{matrix}\right.\]
Let $E_n$ be the event that there are exactly $\lfloor d\ln n\rfloor$ positive locations in $[0,b_2\ln n]$ while the remaining $ \lfloor b_2\ln n\rfloor-\lfloor d\ln n\rfloor$ locations are negative or neutral. If $\min\{a^+,a^+_0\}\geq 1$ we have 
\[\lim_{n\to\infty}\frac{\ln P(E_n)}{\ln n \ln p_n^+} = d =C^+(b_2,h)\]
while otherwise
\[P(E_n)\geq e^{\lfloor d\ln n \rfloor(\ln p_n^+-\ln(p_n^-+p_n^0))+\lfloor b_2\ln \rfloor(p_n^-+p_n^0)}\]
and
\[\limsup_{n\to\infty}\frac{\ln P(E_n)}{\ln n\ln p_n^+}	\leq d\Big(1-\min\{a^+,a^+_0\})+b_2\min\{a^+,a^+_0\}\Big)\xrightarrow{\eta\to 0}C^+(b_2,h)\]
Remembering assumptions (\ref{eq:convdistrp}) and (\ref{eq:convdistrmm}) on the convergence in distribution, we see that 
\[P\left(\ln\rho_y\geq\varepsilon^+-\eta\left|\omega_y^0\leq\frac{1}{n},\ln\rho_y>0\right.\right)\to P^+([\epsilon^+-\eta,\infty))=:\gamma_1>0\]
\[P\left(\ln\rho_y\geq-\delta-\eta\left|\omega_y\leq\frac{1}{n}, \ln\rho_y\leq 0\right.\right)\to Q^-([\delta^-+\eta,\infty))=:\gamma_2>0\]
Therefore, conditional on $E_n$, the event
\[\left\{\sum_{i=0}^{\lfloor b_2\ln n\rfloor}\ln\rho_i\geq d (\varepsilon^+-\eta)\ln n-(b_2-d)(\delta+\eta)\ln n\geq h\ln n\right\}\]
has probability at least $\gamma_1^{\lfloor d\ln n\rfloor}\gamma_2^{\lfloor b_2\ln n\rfloor-\lfloor d\ln n\rfloor}=:c_8^{\ln n}$. Now in total we have 
\[\limsup_{n\to\infty}\frac{\ln P(A_n^+\cap B_n^+)}{\ln n\ln p_n^+}\leq \limsup_{n\to\infty}\frac{\ln P(E_n)+\ln P(A_n^+\cap B_n^+|E_n)}{\ln n\ln p_n^+}\leq \]
\[\leq C^+(b_2,h)+\limsup_{n\to\infty}\frac{c_8}{\ln p_n^+}=C^+(b_2,h)\]
The same reasoning can be applied to $P(A_n^-\cap B_n^-)$, and it remains to show equation (\ref{eq:lemma2res1}). We have 
\[\lim_{n\to\infty}\frac{\ln p_n^+}{\min\{\ln p_n^+,\ln p_n^-\}} =  \frac{1}{\max \{1,a^+\} } = \min\{1,a^-\}\]
\[\lim_{n\to\infty}\frac{\ln p_n^-}{\min\{\ln p_n^+,\ln p_n^-\}} =  \frac{1}{\max \{1,a^-\} } = \min\{1,a^+\}\]
and therefore
\[\lim_{n\to\infty} \frac{\ln [P(A_n^+\cap B_n^+)P(A_n^-\cap B_n^-)]}{\ln n\min\{\ln p_n^+,\ln p_n^-\}}=\]
\[=\lim_{n\to\infty} \frac{\ln P(A_n^+\cap B_n^+)}{\ln n\ln p_n^+}\frac{\ln p_n^+}{\min\{\ln p_n^+,p_n^-\}}+\frac{\ln P(A_n^-\cap B_n^-)}{\ln p_n^-}\frac{\ln p_n^-}{\min\{\ln p_n^+,\ln p_n^-\}}=\]
\[=C^+(b_2,h) \min\{1,a^-\}+C^-(b_1,h)\min\{1,a^+\}=C(b_1,b_2,h)\]
\end{proof}
\begin{proof}[Proof of Theorem \ref{thm2}]
We adapt the steps of the previous Theorem. First for the lower bound we choose $b_1,b_2>0$ arbitrarily and estimate
\[\mathbb{P}^0(\tau>n)\geq P(T_n(0,b_2,b_1,1))\inf_{\omega\in T_n(0,b_1,b_2,1)}\Big\{P_{\omega}(U_n\geq n )P_\omega(\tau>n|U_n>n)\Big\}\]
where $U_n$ is the first time the random walk $X_n$ hits $\{\lfloor-b_1\ln n\rfloor,\lfloor b_2\ln n\rfloor\}$. We can not directly apply lemma \ref{lemmapot1} because $\omega_\cdot^0$ may be non-zero in the interval. Therefore consider the random walk $\overline{X_n}$ conditioned on never staying in one place. That is, we replace $\omega$ by the environment $\overline\omega$ from (\ref{eq:baromega}) and consider a random walk $(\overline{X_n})_{n\in\N}$ with law $P_{\overline{\omega}}$. Obviously $\overline{U_n}\leq U_n$ and so we get (repeating the computation from (\ref{eq:applycor1})) for some $c_9>0$
\[P_{{\omega}}^0({U_n}\geq n)\geq P_{\overline{\omega}}^0(\overline{U_n}\geq n)\geq\exp(-c_9 \ln \ln n)\]
Moreover, on $\{U_n\geq n\}$ the random walk visits only locations with $\omega_\cdot^0\leq\frac{1}{n}$ and therefore for $n$ large enough
\[P_\omega(\tau>n|U_n\geq n)\geq \left(1-\frac{1}{n}\right)^{n}\geq \frac{1}{2e}\]
Since $P(T_n(0,b_1,b_2,1))$ decays faster than $e^{-c_9\ln \ln n}$, this shows
\[\limsup_{n\to\infty}\frac{\ln \mathbb{P}(\tau\geq n)}{\ln n\ln\min\{p_n^+,p_n^-\}}\leq \limsup_{n\to\infty}\frac{\ln \mathbb{P}(\ln P(T_n(x,b_1,b_2,1))}{\ln n\ln\min\{p_n^+,p_n^-\}} =\]
\[=\min\{1,a^-\}C^+(b_2,1)+\min\{1,a^+\}C^-(b_1,1)\]
When taking the infimum over all $b_1,b_2>0$ we see that the exponent is increasing in $b_i$ and that the infimum is attained at \[\overline{b_1}=(\varepsilon^+)^{-1},\quad \overline{b_2} = (\varepsilon^-)^{-1}\]Inserting those values in equation (\ref{eq:lemma2res1}) yields
\[\limsup_{n\to\infty}\frac{\ln \mathbb{P}(\tau\geq n)}{\ln n\ln\min\{p_n^+,p_n^-\}}\leq\inf\{C(b_1,b_2,1):b_i\geq \overline{b_i}\text{ for }i=1,2\}\]
\[=\frac{\min\{1,a^+\}}{\varepsilon^-}+\frac{\min\{1,a^-\}}{\varepsilon^+}=D\]
Now the upper bound: Let $\alpha,\gamma\in (0,1)$ and set $\beta:=\frac{1-\alpha}{\alpha}\frac{1-\gamma}{1+\gamma}$. We say $x$ is \textbf{dangerous} if $\omega_x^0\geq n^{-\beta}$. Moreover, we say $\omega$ is \textbf{good} if
\begin{subequations}\begin{align}
\forall x,y\in[-n,n]: \left\{\begin{matrix}x,y\text{ are dangerous, no location in between is }\\\implies |x-y|\leq \ln^{\kappa+2}n\end{matrix}\right\}\label{eq:newgoodenv1}\\
T_{n^{\alpha}}(x,b_2,b_1,\beta)\text{ does not occur for any }x\in [-n,n],b_2,b_1>0\label{eq:newgoodenv2}
\end{align}\end{subequations}
We need to bound the probability of not being in a good environment: Assume (\ref{eq:newgoodenv1}) is not satisfied. That is, there is an interval of length  at least $\lceil\ln^{2+\kappa}n\rceil$ which does not contain any dangerous location. There are at most $2 n$ such intervals, therefore
\begin{equation}\label{eq:notgood1}P(\text{(\ref{eq:newgoodenv1}) fails for }\omega)\leq 2nP\left(\omega_0^0\leq n^{-\alpha}\right)^{\ln^{2+\kappa}n}\leq 2e^{\ln n+\ln p\ln^{2+\kappa}n}\end{equation}
Here $p=P\left(\omega_0^0\leq n^{-\alpha}\right)$, with $p<1$ for n large enough. Moreover 
\[\lim_{n\to\infty}\frac{\ln P(T_{n^\alpha} (0,b_1,b_2,\beta))}{-\ln^{1+\kappa} n}= \lim_{n\to\infty}\frac{\ln P(T_{n} (0,b_1,b_2,\beta))}{-\left(\frac{1}{\alpha}\right)^{1+\kappa}\ln^{1+\kappa} n} = \alpha^{\kappa+1} C(b_1,b_2,\beta)\]
and so for $n$ large enough we get (using the definition of $\beta$)
\[P(\omega\text{ not good})\leq e^{-\alpha^\kappa(1-\alpha) \frac{1-2\gamma}{1+\gamma} D \ln^{1+\kappa} n}\]
From now on we again consider a good environment $\omega$. Let $a$ and $b$ be dangerous locations in $[-n,n]$ such that no location in $(a,b)$ is dangerous. Then by (\ref{eq:newgoodenv2}) and (\ref{eq:newgoodenv1}) we have
\[H:=H(a+1,b-1)\leq \beta\ln n^\alpha=\alpha\beta\ln n\]
\[\text{and } \ \ \ b-a\leq \beta^{2+\kappa}\ln^{2+\kappa} n\]
Let $\overline{U}=\overline U_{a,b}$ be the first time a random walk $\overline {X_n}$ in the environment $\overline\omega$ from (\ref{eq:baromega}) hits $\{a,b\}$. We can conclude as in (\ref{eq:trapprob}) that for $n$ large enough
\[P_{\overline{\omega}}^x\left(\overline{U}_{a,b}\leq \frac{n^{\alpha\beta(1+\gamma)}}{2}\right)=1-P_{\overline{\omega}}^x\left(\frac{\overline{U}_{a,b}}{\gamma_1 (b-a)^4e^{H}}\geq \frac{n^{\alpha\beta(1+\gamma)}}{2\gamma_1 (b-a)^4e^{H}}\right)\geq\]
\begin{equation}\label{eq:lastlast}  1-\exp\left(-\frac{n^{\alpha\beta(1+\gamma)}}{2\gamma_1 (b-a)^4 e^{H}}\right)\geq 1-\exp\left(-\frac{1}{2\gamma_1\beta^{2+\kappa}} \frac{n^{\alpha\beta\gamma}}{\ln^{4(2+\kappa)} n}\right)\geq \frac{1}{2}\end{equation}
In the same way as in the previous proof we obtain a bound on $U_{a,b}$ from the one on $\overline{U}_{a,b}$, so that
\[P_\omega^x\left({U}_{a,b}\leq {n^{\alpha\beta(1+\gamma)}}\right)\geq \frac{1}{4}\]
Consider $\lfloor n^{1-\alpha\beta(1+\gamma)}\rfloor$ time intervals of length $n^{\alpha\beta(1+\gamma)}$, and let $Z$ be the number of time intervals during which the random walk hits a dangerous location. Then 
\[P\left(Z\leq\frac{1}{8}n^{1-\alpha\beta(1+\gamma)} \right)\leq e^{-c_{10}\;n^{1-\alpha\beta(1+\gamma)}}\]
Let $E_n:=\big\{Z>\frac{1}{8}n^{1-\alpha\beta(1+\gamma)}\big\}$ be the complementary event. On $E_n$, the random walk 
hits at least $\lfloor\frac{1}{8}n^{1-\alpha\beta(1+\gamma)}\rfloor$ dangerous locations and the process survives with probability at most
\[\left(1-\frac{r}{n^\alpha}\right)^{\frac{1}{8}n^{1-\alpha\beta(1+\gamma)}}\sim e^{-\frac{r}{8} n^{1-\alpha\beta(1+\gamma)-\alpha}}\]
Note $1-\alpha\beta(1+\gamma)-\alpha=\gamma (1-\alpha)>0$, and therefore we get
\[\mathbb{P}(\tau>n)\leq P(\omega\text{ not good})+\sup_{\omega\text{ good}}\Big\{P_\omega(\tau>n)\Big\}\]
\[\leq P(\omega\text{ not good})+\sup_{\omega\text{ good}} \Big\{P_\omega(E_n^c)+P_\omega\left(\tau>n\left|E_n\right.\right)\Big\}\]
\begin{equation}\label{eq:last}\leq e^{-\alpha^\kappa(1-\alpha) \frac{1-2\gamma}{1+\gamma} D \ln^{1+\kappa} n}+e^{-c_{10}\; n^{1-\alpha\beta(1+\gamma)}}+e^{-\frac{r}{8} n^{\gamma(1-\alpha)}}\end{equation}
This proves
\[\liminf_{n\to\infty} \frac{\ln\mathbb{P}(\tau>n)}{-c\ln^{1+\kappa} n} \geq \alpha^{\kappa}(1-\alpha)\frac{1-\gamma}{1+\gamma} D\]
\[\xrightarrow{\gamma\to 0} \alpha^{\kappa}(1-\alpha)D\]
Now some easy calculations show that the maximum over all $\alpha$ is attained at $\alpha=\frac{\kappa}{1+\kappa}$, resulting in the lower bound in equation (\ref{eq:thm2res1}).
\end{proof}
\section{Proofs - the stretched exponential case}
\begin{proof}[Proof of Theorem \ref{thm3}]
For the lower bound, let $\alpha,\beta\in (0,1)$ and $\gamma>0$. Then for $n$ large enough we have by lemma \ref{lemma2}
\[P(T_{n^\alpha}(0,\overline{b_1},\overline{b_2},\beta))\geq e^{-(D+\gamma)\alpha\beta n^{\alpha\kappa} \ln n}\]
As in the previous proofs, we introduce a random walk $\overline{X_n}$ in an environment $\overline{\omega}$ without holding times, and denote by $\overline{U_n}$ the first time this random walk leaves the interval $[-\overline{b_2}\ln n,\overline{b_1}\ln n]$. Now we can use corollary \ref{korrpot}, so that
\[P_{\overline{\omega}}\left(\;\overline{U_n}>n^{\alpha\beta}\right)\geq\frac{1}{2(\overline{b_1}+\overline{b_2})\ln n}\exp\left(-\frac{\gamma_3\ln(2(\overline{b_1}+\overline{b_2})\ln n))n^{\alpha\beta}}{n^{\alpha\beta}}\right)=:e^{-c_{11}\ln \ln n}\]
As before, we have $P_\omega(U_n>n^{\alpha\beta})\geq P_{\overline\omega}(\;\overline {U_n}>n^{\alpha\beta})$, and conditional on the event $T_n^\alpha(0,\overline{b_1},\overline{b_2},\beta)\cap\{U_n>n^{\alpha\beta}\}$, the walk visits only locations $x$ with $\omega_x^0\leq n^{-\alpha}$ until time $n^{\alpha\beta}$. In this case the probability for survival until $n^{\alpha\beta}$ is at least
\[\left(1 -\frac{r}{n^\alpha}\right)^{n^{\alpha\beta}}\geq e^{-\frac{r}{2} n^{\alpha\beta-\alpha}}\geq \frac{1}{2}\]
Now conditional on $\tau>n^{\alpha\beta}$, the random walk survives the remaining $n^{1-\alpha\beta}$ steps with probability at least $(c_{12})^{n^{1-\alpha\beta}}$, where $c_{12}:=(1-r)(1-2\varepsilon_0)$ is due to (\ref{eq:defue}). In total we have shown
\[\mathbb{P}(\tau>n)\geq P(T_n^\alpha(0,\overline{b_1},\overline{b_2},\beta))\]
\[\times\inf_{\omega\in T_n^\alpha(0,\overline{b_1},\overline{b_2},\beta)}\Big\{ P_\omega(U_n>n^{\alpha\beta})P_\omega(\tau>n^{\alpha\beta}|U_n>n^{\alpha\beta})P_\omega(\tau>n|\tau>n^{\alpha\beta})\Big\}\]
\[\geq \frac{1}{4} e^{-(D+\gamma)\alpha\beta n^{\alpha\kappa} \ln n +\ln c_{12} n^{1-\alpha\beta}}e^{c_{11}\ln \ln n}\]
We want to minimize $\max\{\alpha\kappa,1-\alpha\beta\}$ subject to $\alpha,\beta\in (0,1)$. Setting $\beta:=\frac{1-\alpha\kappa}{\alpha}$ this maximum is equal to $\alpha\kappa$, where $\alpha>\frac{1}{1+\kappa}$. This proves
\[\limsup_{n\to\infty} \frac{\ln (-\ln \mathbb{P}(\tau>n))}{\ln n} \leq \alpha\kappa\xrightarrow{\alpha\downarrow\frac{1}{1+\kappa}}\frac{\kappa}{1+\kappa}\]
Now concerning the upper bound: Let $\alpha,\beta\in(0,1)$, $\gamma\in(0,\frac{1-\beta}{4})$ and $\delta>0$. We say $x\in\mathbb{Z}$ is \textbf{safe} if $\omega_x^0\leq n^{-\alpha}$. We call $\omega$ \textbf{good} if
\begin{subequations}\begin{align}T_{n^{\alpha}}(x,b_2,b_1,\beta)\text{ does not occur for any }x\in [-n,n],b_2,b_1>0\label{eq:newnewgoodenv2}\\
\forall x,y\in[-n,n]: \left\{\begin{matrix}x,y\text{ are dangerous, no location in between is }\\\implies |x-y|\leq n^\gamma\end{matrix}\right\}\label{eq:newnewgoodenv1}\end{align}\end{subequations}
By the same reasoning as in the previous proofs,
\[P(\text{(\ref{eq:newnewgoodenv2}) does not occur})\leq e^{c_{13}\beta \ln n^\alpha n^{\alpha\kappa}}\]
so that
\[P(\omega\text{ is not good})\leq e^{-c_{14}\; n^\gamma} + e^{-c_{13}\beta\alpha\ln(n) n^{\alpha\kappa}}\]
Choosing as before an interval $[a,b]\subset\Z$ with 
\[[a,b]\cap\{\text{safe locations}\}=\{a,b\}\]
and modifying equation (\ref{eq:lastlast}) yields
\[P_{\overline{\omega}}^x\left(\overline{U}_{a,b}\leq \frac{n^{\alpha\beta+4\gamma+\delta}}{2}\right)\geq 1-\exp\left(-\frac{n^{\alpha\beta+4\gamma+\delta}}{2\gamma_1 n^{4\gamma} n^{\alpha\beta}}\right)\geq \frac{1}{2}\]
Again we easily convert this into a bound for the random walk with holding times
\[P_{{\omega}}^x\left({U}_{a,b}\leq {n^{\alpha\beta+4\gamma+\delta}}\right)\geq\frac{1}{4}\]
Let $Z$ be the number of time intervals of length $n^{\alpha\beta+4\gamma+\delta}$ during which the random walk hits a dangerous location, and $E_n:=\big\{Z> \frac{1}{8}n^{1-4\gamma-\alpha\beta-\delta}\big\}$. Then
\[P(E_n^c)\leq e^{-c_{15} n^{1-4\gamma-\alpha\beta-\delta}}\]
On the other hand, on $E_n$ the probability of surviving until $n$ is at most
\[\left(1-\frac{r}{n^\alpha}\right)^{\frac{1}{8}n^{1-4\gamma-\alpha\beta-\delta}}\leq e^{-\frac{r}{16} n^{1-4\gamma-\alpha\beta-\delta-\alpha}}\]
For this to make sense we assume that $\beta$ and $\delta$ are small enough such that the exponent in the last equation is positive. In total we get
\[\mathbb{P}(\tau>n)\leq e^{-c_{14} n^\gamma} + e^{-c_{13}\alpha\beta \ln(n) n^{\alpha\kappa}} + e^{-c_{15} n^{1-4\gamma-\alpha\beta-\delta}} + e^{-\frac{r}{16} n^{1-4\gamma-\alpha\beta-\delta-\alpha}}\]
and now have to find $\alpha,\beta,\gamma,\delta$ such that \[\min\{\gamma,\alpha\kappa,1-4\gamma-\alpha\beta-\delta,1-4\gamma-\alpha\beta-\delta-\alpha\}\] becomes maximal. Disregarding the third term, setting $\alpha\kappa=\gamma$ and $\alpha\kappa = 1-4\gamma-\alpha\beta-\delta-\alpha$ we get
\[\alpha =\frac{1-\delta}{1+5\kappa+\beta}\]
and therefore
\[\liminf_{n\to\infty}\frac{\ln(-\ln\mathbb{P}(\tau>n))}{\ln n}\geq \frac{\kappa(1-\delta)}{1+5\kappa+\beta}\xrightarrow{\beta,\delta\downarrow 0} \frac{\kappa}{1+5\kappa}\]
\end{proof}
\section*{Acknowledgements}
We would like to thank Nina Gantert for several inspiring conversations. We also want to thank an anonymous referee for pointing out several mistakes in a preliminary version. 
\bibliographystyle{alea3}
\bibliography{template}
\end{document}